\documentclass{amsart}
\usepackage{amssymb,amsthm,amsmath,prooftree,graphics}
\usepackage[all]{xy}
\usepackage[latin1]{inputenc}
\input{diagxy}
\newcommand{\judge}[2]{#1\;\vdash\;#2}
\newcommand{\groupoids}{\textnormal{\textbf{Gpd}}}

\newcommand{\topcat}{\textnormal{\textbf{Top}}}
\newcommand{\ssets}{\textnormal{\textbf{SSet}}}
\newcommand{\id}[1]{\textnormal{Id}_{#1}}

\newcommand{\type}{\operatorname{type}}
\newcommand{\app}{\operatorname{app}}

\renewcommand{\implies}{\Rightarrow}

\newcommand{\iso}{\cong}
\newtheorem{theorem}{Theorem}[section]

\newtheorem{proposition}[theorem]{Proposition}
\newtheorem{corollary}[theorem]{Corollary}

\theoremstyle{definition}

\theoremstyle{remark}

\begin{document}
\author{Steve Awodey and Michael A. Warren}
\email{awodey@cmu.edu}
\email{mwarren@andrew.cmu.edu}
\date{\today}
\address{Department of Philosophy\\Carnegie Mellon
  University\\Pittsburgh, PA\\USA 15213}

\title[identity types]{Homotopy Theoretic Models of Identity Types}
\maketitle
\section{Introduction}

\noindent Quillen \cite{Quillen:HA} introduced model categories as an
abstract framework for homotopy theory which would apply to a wide
range of mathematical settings.  By all accounts this program has been
a success and ---  as, e.g., the work of Voevodsky on the homotopy theory of
schemes \cite{Morel:A1HTS} or the work of Joyal
\cite{Joyal:QCKC,Joyal:NOQC} and Lurie \cite{Lurie:HTT} on
quasicategories seems to indicate --- it will likely continue to
facilitate mathematical advances.  In this paper we present a
novel connection between model categories and mathematical logic, inspired by
the groupoid model of (intensional) Martin-L\"{o}f type theory
\cite{MartinLof:ITT} due to Hofmann and Streicher \cite{Hofmann:GITT}.  In
particular, we show that a form of Martin-L\"{o}f type
theory can be soundly modelled in any model
category.  This result indicates moreover that any model category
has an associated ``internal language'' which is itself a form of
Martin-L\"{o}f type theory.  This suggests applications both to
type theory and to homotopy theory.  Because Martin-L\"{o}f type
theory is, in one form or another, the theoretical basis for many of
the computer proof assistants currently in use, such as \emph{Coq} and
\emph{Agda} (cf.~\cite{Bertot:ITPPD} and \cite{Coquand:STT}),
this promise of applications is of a practical, as well as theoretical, nature.

The present paper provides a precise indication of this connection
between homotopy theory and logic; a more detailed discussion
of these and further results will be given in \cite{Warren:PhD}.

\section{Type Theory}\label{section:type_theory}

Type theory is concerned with (at least) two basic kinds of
entities: \emph{types} and \emph{terms}.  Types are
written as $A,B,\ldots$ and terms as $a,b,\ldots$.  Every
term has a unique type and we write $a:A$ to indicate that $a$ is a
term of type $A$.  Types can be thought of as sets and terms as
elements of sets or, respectively, as objects of a category and global sections
thereof.  Alternatively, under an interpretation known as the
Curry-Howard correspondence (cf.~\cite{Nordstrom:PMLTT}), a type $A$ can be
regarded as a proposition and a term $a:A$ as a proof of $A$.  

The \emph{simply typed $\lambda$-calculus} is the type theory obtained
by admitting the construction of products $(A\times B)$ and
exponentials (function
spaces) $(A\rightarrow B)$ of types $A$ and $B$.  Under the
Curry-Howard correspondence, the simply typed $\lambda$-calculus
describes the behavior of proofs in propositional (intuitionistic) logic:
$(A\times B)$ is the conjunction $(A\wedge B)$ and $(A\rightarrow
B)$ is the implication $(A\implies B)$.
In categorical terms, the simply typed $\lambda$-calculus corresponds
to cartesian closed categories in the evident way.

The principal innovation of Martin-L\"{o}f's \emph{dependent type theory}
over the simply typed $\lambda$-calculus is that types are allowed to
\emph{depend on} or ``vary over'' other types, thereby yielding a
more complex and expressive theory.  The meaning of type
dependence is that, when $A$ is a given type, it is possible for a family
$(B_{x})_{x:A}$ of types to occur indexed by $A$.  The theory also allows
families of types which are themselves indexed by families of types,
and so forth.  The basic operations of the theory then correspond to
indexed sums and products.  These operations, together with type
dependence, allow us to regard dependent type theory as an extension
of the Curry-Howard correspondence to first-order (intuitionistic)
logic.  Similarly, the kinds of categories corresponding to dependent
type theory are locally cartesian closed categories.

We now present the syntax of Martin-L\"{o}f type theory in more
detail together with an interpretation, due to Seely
\cite{Seely:LCCCTT}, in locally cartesian closed categories.  This
interpretation is ``non-split'' in the sense that it does not model
substitution on the nose, but only up to canonical natural
isomorphism, due to the pseudo-functoriality introduced by a choice of
pullbacks (cf.~\cite{Curien:SUI} and \cite{Hofmann:OITTLCCC}).
Because we are mostly interested in type theory as an internal
language for categories this conflation of isomorphic objects will not
concern us here.  The homotopy theoretical interpretation will be
given in the Section \ref{section:models}.

\subsection{Forms of judgement}

The syntax of type theory is given by first indicating four ``forms  
of judgement''.  These are the basic kinds of statement which can be
formally made in the theory.  The first form of judgement is the type
declaration $\judge{}{A:\type}$ which says that $A$ is a type.  In a
fixed locally cartesian closed category $\mathcal{C}$ such a judgement
is interpreted as an object $A$ of $\mathcal{C}$.  As mentioned above,
when $A$ is a type it is possible to consider
$A$-indexed families of types.  That $B(x)$ is an $A$-indexed family
of types is indicated by the following form of judgement
\begin{align}\label{eq:type_dep}  
  \judge{x:A}{B(x):\type}.  
\end{align}
Such a judgment is interpreted as an arrow $f:B\to<150>A$ 
with codomain $A$ following the usual categorical treatment of indexed
families.

In (\ref{eq:type_dep}) the part $x:A$
to the left of the turnstile $\;\vdash\;$ is called the \emph{context} of the
judgement.  More generally, a list of variable declarations
\begin{align*}
  x_{0}:A_{0},x_{1}:A_{1},\ldots,x_{n}:A_{n}
\end{align*}
is a context whenever the judgements $\judge{}{A_{0}:\type}$ and 
\begin{align*}
  \judge{x_{0}:A_{0},\ldots,x_{m}:A_{m}}{A_{m+1}:\type}
\end{align*}
are derivable for $0\leq m<n$.  Upper-case Greek letters
$\Gamma,\Delta,\ldots$ are reserved as names for contexts.  Contexts
are interpreted in the natural way as chains 
\begin{align}\label{eq:chain}
  A_{n}\to<150>A_{n-1}\to<150>\cdots\to<150>A_{0}
\end{align}
of arrows.  The empty context is interpreted as the terminal object.

In addition to judgements of the form $\judge{\Gamma}{A:\type}$ there
are also judgements of the form
\begin{align}\label{eq:term_in_context}
  \judge{\Gamma}{a:A},
\end{align}
which state that $a$ is a \emph{term} of type $A$ in the context
$\Gamma$.  In the empty context a term $a:A$ is interpreted as a
global section $1\to<150>A$ of the object $A$.  Similarly, when $\Gamma$ is
interpreted as a chain of arrows of the form (\ref{eq:chain}) the
judgement (\ref{eq:term_in_context}) is interpreted as a section $a:A_{n}\to<150>A$ of the interpretation $A\to<150>A_{n}$ of $\judge{\Gamma}{A:\type}$.

Finally, there are also forms of judgement governing \emph{definitional
  equality} of types and terms as follows:
\begin{align*}
  &\judge{\Gamma}{A=B:\type},\\
  &\judge{\Gamma}{a=b:A},
\end{align*}
which are interpreted as identities in $\mathcal{C}$.  Henceforth,
when no confusion will result, explicit mention of contexts
will be elided.

\subsection{Dependent sums and products}

Given an $A$-indexed family of types $B(x)$ the \emph{dependent sum}
$\Sigma_{x:A}.B(x)$ and the \emph{dependent product} $\Pi_{x:A}.B(x)$
can be formed.  This is usually stated as the following
\emph{formation rules} 
\begin{align*}
  \begin{prooftree}
    \judge{x:A}{B(x)}
    \justifies
    \judge{}{\Sigma_{x:A}.B(x):\type}
    \using{\Sigma\text{ form.}}
  \end{prooftree}
  \qquad\text{ and }\qquad
  \begin{prooftree}
    \judge{x:A}{B(x)}
    \justifies
    \judge{}{\Pi_{x:A}.B(x):\type}
    \using{\Pi\text{ form.}}
  \end{prooftree}
\end{align*}
Under the Curry-Howard correspondence, dependent sums correspond to
existential quantifiers and dependent products correspond to universal
quantifiers.  The behavior of these types is specified by
\emph{introduction}, \emph{elimination} and \emph{conversion} rules,
which can be thought of either in terms of manipulation of indexed
families or their logical significance.
For example, the introduction rule for $\Pi_{x:A}.B(x)$ is stated as
\begin{align*}
  \begin{prooftree}
    \judge{x:A}{f(x):B(x)}
    \justifies
    \judge{}{\lambda_{x:A}.f(x):\Pi_{x:A}.B(x)}
    \using{\Pi\text{ intro.}}
  \end{prooftree}
\end{align*}
which states that if $f$ is family of terms $f(x):B(x)$, then there is
a term $\lambda_{x:A}.f(x)$ of type $\Pi_{x:A}.B(x)$.  Similarly, the elimination rule
\begin{align*}
  \begin{prooftree}
    \judge{}{g:\Pi_{x:A}.B(x)}\quad\judge{}{a:A}
    \justifies
    \judge{}{\app(g,a):B(a)}
    \using{\Pi\text{ elim.}}
  \end{prooftree}
\end{align*}
corresponds to the application of an element $g$ of the indexed
product to $a:A$.  Finally, the following conversion rule for
dependent products states that the application term $\app(g,a)$
behaves correctly when $g$ is itself of the form $\lambda_{x:A}.f(x)$:
\begin{align*}
  \begin{prooftree}
    \judge{x:A}{f(x):B(x)}\quad\judge{}{a:A}
    \justifies
    \judge{}{\app\bigl(\lambda_{x:A}.f(x),a\bigr)\;=\;f(a):B(a)}
    \using{\Pi\text{ conv}.}
  \end{prooftree}
\end{align*}
The dependent sums $\Sigma_{x:A}.B(x)$ are likewise
required to obey suitable introduction, elimination and conversion
rules.  When types $A$ and $B$ do not depend on any variables, the
usual product type $(A\times B)$ and exponential type $(A\rightarrow B)$ from the simply typed $\lambda$-calculus
are recovered as $\Sigma_{x:A}.B$ and $\Pi_{x:A}.B$, respectively.

In a locally cartesian closed category $\mathcal{C}$, the dependent
products and sums are interpreted in the natural way using,
respectively, the right and left adjoints to the pullback functors.

\subsection{Identity types}

In addition to dependent sums and products it is required that 
for each type $A$ and terms $a,b:A$,
there exists a type $\id{A}(a,b)$ called the \emph{identity type}
which provides the only explicit form of type dependence in the theory
considered here.  I.e., unlike dependent products and sums, the formation rule
for the identity type introduces new type dependencies:
\begin{align}
  \label{eq:id_form}
  \begin{prooftree}
    \judge{}{a:A}\qquad\judge{}{b:A}
    \justifies
    \judge{}{\id{A}(a,b):\type}
    \using{\id{}\text{ form.}}
  \end{prooftree}
\end{align}
Under the Curry-Howard correspondence, this type is regarded
as the proposition which states that $a$ and $b$ denote identical
proofs of the proposition $A$.  The introduction rule
\begin{align}
  \label{eq:id_intro}
  \begin{prooftree}
    \judge{}{a:A}
    \justifies 
    \judge{}{r_{A}(a):\id{A}(a,a)}
    \using{\id{}\text{ intro.}}
  \end{prooftree}
\end{align}
states that given a term $a:A$ there is always a witness $r_{A}(a)$ to
the proposition that $a$ is identical to itself.  We call $r_{A}(a)$
the \emph{reflexivity term}.  On the other hand, the distinctive elimination rule 
\begin{align}
  \label{eq:id_elim}
  \begin{prooftree}
    \[
    \judge{x:A,y:A,z:\id{A}(x,y)}{D(x,y,z):\type}
    \thickness.0em
    \justifies
    \judge{}{p:\id{A}(a,b)}
    \qquad\qquad
    \judge{x:A}{d(x):D\bigl(x,x,r_{A}(x)\bigr)}
    \]
    \justifies
    \judge{}{J_{A,D}(d,a,b,p):D(a,b,p)}
    \using{\id{}\textnormal{ elim.}}
  \end{prooftree}
\end{align}
can be recognized as a form of Leibniz's law.  Finally, the conversion rule
\begin{align}
  \label{eq:id_conv}
   \begin{prooftree}
     \[
     \judge{x:A,y:A,z:\id{A}(x,y)}{D(x,y,z):\type}
     \thickness.0em
     \justifies
     \judge{}{a:A}
     \qquad\qquad
     \judge{x:A}{d(x):D\bigl(x,x,r_{A}(x)\bigr)}
     \]
     \justifies
     \judge{}{J_{A,D}(d,a,a,r_{A}(a))=d(a):D(a,a,r_{A}(a))}
     \using{\id{}\textnormal{ conv.}}
   \end{prooftree}
\end{align}
indicates that the elimination term is equal to
$d(a)$ when $p$ is the reflexivity term.

\subsection{Locally cartesian closed categories are extensional}

A model of Martin-L\"{o}f type theory is
\emph{extensional} if the following reflection
rule is satisfied:
\begin{align}\label{eq:reflection}
  \begin{prooftree}
    \judge{}{p:\id{A}(a,b)}
    \justifies
    \judge{}{a=b:A}.
    \using{\id{}\text{ refl.}}
  \end{prooftree}
\end{align}
I.e., the identity type
$\id{A}(a,b)$ captures no more information than whether or not $a$ and
$b$ are definitionally equal.  Although type checking is decidable in
the intensional theory, it fails to be in the extensional theory
obtained by adding (\ref{eq:reflection}) as a rule governing identity types.
This fact is the principal motivation for studying intensional rather
than extensional type theories (cf.~\cite{Streicher:IIITT} for a more
thorough discussion of the phenomenon of intensionality and the
difference between intensional and extensional forms of the theory).
Under the general interpretation in
locally cartesian closed categories
sketched above the reflection rule is always valid.
\begin{proposition}
  In the standard interpretation given above, every locally cartesian
  closed category $\mathcal{C}$ is extensional. 
\end{proposition}
\begin{proof}
  Note that it suffices to consider
  ``parameterized'' versions of the rules governing identity types.
  I.e., the rules given above
  are equivalent, by the structural rules of the theory, to the rules
  obtained by replacing any terms $a,b:A$ and $p:\id{A}(a,b)$ by
  variables $x,y:A$ and $z:\id{A}(x,y)$, and stating judgements in
  the appropriate context.  E.g., (\ref{eq:id_elim}) is equivalent to
  \begin{align*}
    \begin{prooftree}
      \[\judge{x:A,y:A,z:\id{A}(x,y)}{D(x,y,z):\type}
      \thickness.0em
      \justifies
      \judge{x:A}{d(x):D\bigl(x,x,r_{A}(x)\bigr)}
      \]
      \justifies
      \judge{x,y:A,z:\id{A}(x,y)}{J_{A,D}(d,x,y,z):D(x,y,z).}
    \end{prooftree}
  \end{align*}
  As such, it suffices to prove that, when $A$ is an object of $\mathcal{C}$,
  any object $\id{A}$ satisfying the introduction, elimination and conversion
  rules for the identity type is isomorphic to the diagonal
  $\Delta:A\to<150>A\times A$.  By the formation and introduction
  rules (\ref{eq:id_form}) and (\ref{eq:id_intro}), there exists a factorization
  \begin{align}
    \label{eq:intro_and_form}
    \bfig
    \Vtriangle<350,350>[A`\id{A}`A\times A;r`\Delta`p]
    \efig
  \end{align}
  of the diagonal.  In the interpretation, $r$ may itself be regarded as a type over
  $\id{A}$.  By (\ref{eq:intro_and_form}), this type satisfies the
  hypotheses of the elimination rule and therefore there exists a section
  $J:\id{A}\to<150>A$ of $r$,
  \begin{align*}
    \bfig
    \Vtriangle|blr|<350,350>[A`\id{A}`A\times A,;r`\Delta`p]
    \morphism(700,350)/{@{..>}@/_1em/}/<-700,0>[\id{A}`A;J]
    \efig
  \end{align*}
  as required.
\end{proof}
We now consider homotopy
models of type theory, which do not validate the reflection rule.

\section{Homotopy Theoretic Models}\label{section:models}

In order to obtain models of type theory which do not validate the
reflection rule additional higher-dimensional structure must be
considered in the interpretation.  One way to add such structure is
via the device of weak-factorization systems and
Quillen model categories (cf.~\cite{Quillen:HA} and \cite{Bousfield:CFSC}).

\subsection{Weak factorization systems}

In any category $\mathcal{C}$, given maps $f:A\to<150>B$ and $g:C\to<150>D$, we write 
\begin{align*}
  f\pitchfork g
\end{align*}
to indicate that $f$ has \emph{left-lifting property} (\emph{LLP}) with respect to
$g$.  I.e. for any commutative square
\begin{align*}
  \bfig
  \square<400,350>[A`C`B`D;h`f`g`k]
  \morphism(0,0)|m|/..>/<400,350>[B`C;l]
  \efig
\end{align*}
there exists a map $l:B\to<150>C$ such that $g\circ l=k$ and $l\circ
f=h$.  Similarly, if $\mathfrak{M}$ is any collection of maps we denote by
$^{\pitchfork}\mathfrak{M}$ the collection of maps in
$\mathcal{C}$ having the LLP with respect to all maps in $\mathfrak{M}$.
The collection of maps $\mathfrak{M}^{\pitchfork}$ is defined similarly.

A \emph{weak factorization system} $(\mathfrak{L},\mathfrak{R})$ in a category
$\mathcal{C}$ consists of two collections $\mathfrak{L}$ (the
``left-class'') and
$\mathfrak{R}$ (the ``right-class'') of maps in $\mathcal{C}$ such that 
\begin{enumerate}
\item Every map $f:A\to<150>B$ has a factorization as
  \begin{align*}
    \bfig
    \Vtriangle<300,300>[A`C`B;i`f`p]
    \efig
  \end{align*}
  where $i$ is a member of $\mathfrak{L}$ and $p$ is a member of $\mathfrak{R}$.
\item $\mathfrak{L}^{\pitchfork}=\mathfrak{R}$ and $\mathfrak{L}=\;^{\pitchfork}\mathfrak{R}$.
\end{enumerate}

\subsection{Model categories}

A (\emph{closed}) \emph{model category} \cite{Quillen:HA} is a bicomplete category $\mathcal{C}$
equipped with subcategories $\mathfrak{F}$ (\emph{fibrations}), $\mathfrak{C}$
(\emph{cofibrations}) and $\mathfrak{W}$ (\emph{weak equivalences})
satisfying the following two conditions:
\begin{enumerate}
\item (``Three-for-two'') Given a commutative triangle
  \begin{align*}
    \bfig
    \Vtriangle<300,300>[A`B`C;f`h`g]
    \efig
  \end{align*}
  if any two of $f,g,h$ are weak equivalences, then so is the third.
\item Both $(\mathfrak{C},\mathfrak{F}\cap\mathfrak{W})$ and
$(\mathfrak{C}\cap\mathfrak{W},\mathfrak{F})$ are weak factorization systems.
\end{enumerate}
A map $f$ is an \emph{acyclic cofibration} if it is in
$\mathfrak{C}\cap\mathfrak{W}$, i.e. both a
cofibration and a weak equivalence.  Similarly, an \emph{acyclic
  fibration} is a map in $\mathfrak{F}\cap\mathfrak{W}$, i.e. which is
simultaneously a fibration and a weak
equivalence.  An object $A$ is said to be \emph{fibrant} if the
canonical map $A\to<150>1$ is a fibration.  Similarly, $A$ is
\emph{cofibrant} if $0\to<150>A$ is a cofibration.

Examples of model categories include the following:
\begin{enumerate}
\item The category $\topcat$ of topological spaces with fibrations the
  \emph{Serre fibrations}, weak equivalences the weak homotopy
  equivalences and cofibrations those maps which have the LLP with
  respect to acyclic fibrations.  The cofibrant objects in this model
  structure are retracts of spaces constructed, like CW-complexes, by
  attaching cells.
\item The category $\ssets$ of simplicial sets with cofibrations the
  monomorphisms, fibrations the Kan fibrations and weak equivalences
  the weak homotopy equivalences.  The fibrant objects for this model
  structure are the Kan complexes.
\item The category $\groupoids$ of (small) groupoids with cofibrations
  the functors injective on objects, fibrations the Grothendieck
  fibrations and weak equivalences the categorical equivalences.  Here
  all objects are both fibrant and cofibrant.
\end{enumerate}
The reader should consult, e.g., \cite{Hovey:MC} or \cite{Dwyer:HTMC}
for further examples and details.

\subsection{Path objects}

Recall from \cite{Hovey:MC}, that in a model category $\mathcal{C}$ a (\emph{very good}) \emph{path
  object} $A^{I}$ for an object $A$ consists of a factorization 
\begin{align*}
  \bfig
  \Vtriangle<300,300>[A`A^{I}`A\times A;r`\Delta`p]
  \efig
\end{align*}
of the diagonal map $\Delta:A\to<150>A\times A$ as an acyclic
cofibration $r$ followed by a fibration $p$.  Paradigm examples of
path objects are given by exponentiation by the
``unit interval'' $I$ in either $\groupoids$ or, when the object $A$
is a Kan complex, in $\ssets$.  In
$\groupoids$, $I$ is the connected groupoid with
exactly two objects (i.e., the ``arrow category'') and in $\ssets$ it
is the $1$-simplex $\Delta[1]$.

Path objects may also be fruitfully considered in the context of weak
factorization systems, where the left class $\mathfrak{L}$ is
thought of as the acyclic cofibrations and the right class
$\mathfrak{R}$ as the fibrations.  In both weak factorization systems
and model categories path objects are guaranteed to exist, but need
not be uniquely determined.  Moreover,
the path object construction is often functorial.

\subsection{The interpretation}

Whereas the idea of the Curry-Howard correspondence is often
summarized by the slogan ``Propositions as Types'', the idea
underlying the interpretation of type theory in weak factorization
systems and model categories is 
\begin{center}
  \emph{Fibrations as Types}.
\end{center}
Specifically, assume that $\mathcal{C}$ is a finitely complete
category with a weak factorization system
$(\mathfrak{L},\mathfrak{R})$.  Because most interesting examples
arise from model categories, we refer to maps in $\mathfrak{L}$ as
acyclic cofibrations and those in $\mathfrak{R}$
as fibrations.  We describe the interpretation in the style of an
``internal language'' for $\mathcal{C}$, as in Section
\ref{section:type_theory} for locally cartesian closed categories.

In such a category $\mathcal{C}$, a judgement $\judge{}{A:\type}$ is
interpreted as a fibrant object $A$ of $\mathcal{C}$.
Similarly, $\judge{x:A}{B(x):\type}$ is interpreted as a
fibration $f:B\to<150>A$.  Contexts are interpreted as chains of
fibrations.  Terms $\judge{\Gamma}{a:A}$ in context are interpreted,
as usual, as sections of the interpretation of
$\judge{\Gamma}{A:\type}$.  

Thinking, in this way, of fibrant objects as types and fibrations as
dependent types, the natural interpretation of the identity type
$\id{A}(a,b)$ should be as the ``fibrant object'' of paths in $A$ from $a$ to
$b$, and $\judge{x,y:A}{\id{A}(x,y):\type}$ should be ``the'' fibrant
object of all paths in $A$.  That is, it should be a path object for $A$.

We now show that this interpretation soundly models a form of
type theory with identity types (see Appendix \ref{section:appendix}
for the details of this theory).  The interpretation of type formers
other than identity types, together with some of the coherence issues
related to the interpretation, is discussed in Section \ref{section:coherence}.
\begin{theorem}\label{theorem:main}
  Let $\mathcal{C}$ be a finitely complete category with a weak
  factorization system and a functorial choice $(-)^{I}$ of path
  objects in $\mathcal{C}$, and
  all of its slices, which is stable under substitution.  I.e., given
  any fibration $B\to<150>A$ and $\sigma:A'\to<150>A$,
  \begin{align*}
    \sigma^{*}\bigl(B^{I}\bigr) & \iso \bigl(\sigma^{*}B\bigr)^{I}.
  \end{align*}
  Then $\mathcal{C}$ is a model of a form of Martin-L\"{o}f type theory
  with identity types.  
\end{theorem}
\begin{proof}
  We may work in the empty context since the relevant structure is
  stable under slicing.  Given such a choice of path objects, we
  interpret, given a fibrant object $A$, the judgement
  $\judge{x,y:A}{\id{A}(x,y)}$ as the path object fibration
  $p:A^{I}\to<150>A\times A$.  Because $p$ is a fibration,
  the formation rule (\ref{eq:id_form}) is satisfied.  Similarly, the introduction rule
  (\ref{eq:id_intro}) is valid because $r:A\to<150>A^{I}$ is a section of $p$.

  For the elimination and conversion rules, assume
  that the following premisses are given 
  \begin{align*} 
    \judge{x:A,y:A,z:\id{A}(x,y)}{D:\type&},\\
    \judge{x:A}{d(x):D(x,x,r_{A}(x))&}. 
  \end{align*} 
  As such, there exists a fibration $g:D\to<150>A^{I}$ together with a
  map $d:A\to<150>D$ such that $g\circ d=r$.  This data yields the
  following commutative square: 
  \begin{align*} 
    \bfig 
    \square<400,350>[A`D`A^{I}`A^{I}.;d`r`g`1] 
    \efig 
  \end{align*} 
  Because $g$ is a fibration and $r$ is, by definition, an acyclic
  cofibration, there exists a diagonal filler.  
  \begin{align}\label{eq:the_diagram} 
    \bfig 
    \square<400,350>[A`D`A^{I}`A^{I}.;d`r`g`1] 
    \morphism(0,0)|m|/..>/<400,350>[A^{I}`D;J] 
    \efig 
  \end{align} 
  Choose such a filler $J$ as the interpretation of the term:
  \begin{align*}
    \judge{x,y:A,z:\id{A}(x,y)}{J_{A,D}(d,x,y,z):D(x,y,z)}.
  \end{align*}
  Commutativity of the bottom triangle of (\ref{eq:the_diagram}) is
  precisely the conclusion of the elimination rule (\ref{eq:id_elim}) and commutativity
  of the top triangle is the conversion rule (\ref{eq:id_conv}). 
\end{proof}
Examples of categories satisfying the hypotheses of Theorem
\ref{theorem:main} include $\groupoids$, $\ssets$ and many
\emph{simplicial model categories} \cite{Quillen:HA} (including, e.g.,
simplicial sheaves and presheaves).  We include a proof of this fact
for the benefit of those readers who are familiar with simplicial
model categories.  This example will be considered in more detail in
\cite{Warren:PhD}.
\begin{corollary}
  Every simplicial model category $\mathcal{C}$ in which $\mathfrak{C}$ is
  the class of monomorphisms satisfies the hypotheses of Theorem
  \ref{theorem:main}, and is therefore a model of intensional type theory.
  \begin{proof}
    Let $I$ be the unit interval $\Delta[1]$ in $\ssets$, and consider, for any
    fibrant object $A$ of $\mathcal{C}$, the factorization of the
    diagonal given by 
    \begin{align*}
      \bfig
      \Vtriangle<350,350>[A`A^{I}`A\times A;r`\Delta`p]
      \efig
    \end{align*}
    where $r$ is the ``constant loop'' map obtained as the transpose,
    under the (enriched) adjunctions involved, of the map 
    $I\to<150>\mathcal{C}[A,A]$ obtained by composing the canonical
    map $I\to<150>1$ with the insertion of identities map
    $1\to<150>\mathcal{C}[A,A]$ and $p$ is the map obtained by
    $A^{I}\to<150>A^{\partial I}$ induced by the inclusion of the
    boundary $\partial I$ into $I$.  Because $\partial I\to<150>I$ is
    a monomorphism and $A$ is fibrant it follows that $p$ is a
    fibration.  Because $r$ is a simplicial homotopy equivalence it is
    also a weak equivalence.  The required pullback stability is seen
    to hold using the adjunctions defining the
    factorization.  Stability under slicing of this choice of
    factorization (as well as the structure defining simplicial model
    categories) is a routine verification.
  \end{proof}
\end{corollary}

\section{Additional Topics}\label{section:coherence}

We now briefly consider the particular features of the type theory occurring as the internal
language of model categories, as well as the connection of this work
with the groupoid model of Hofmann and Streicher \cite{Hofmann:GITT}.
These topics will be addressed fully in \cite{Warren:PhD}. 

\subsection{The internal language of model categories}

The form of type theory to which Theorem \ref{theorem:main} applies
differs from the standard theory presented in, say,
\cite{MartinLof:ITT} in two ways. Namely, because arbitrary model
categories need not be locally
cartesian closed --- or, even if they are, need not have $\Pi$ functors which preserve
fibrations --- such a category may not possess sufficient structure to
interpret dependent products in the standard way.  However, for the purposes of
modelling type theory this is not much of a limitation since most
model categories do possess well behaved $\Pi$ functors.  So, for
example, $\ssets$ as well as most other presheaf model categories
do, \emph{qua} toposes with appropriate model structures, support the
interpretation of dependent products.  Note that the
rules for dependent sums are, trivially, always valid in this
interpretation because fibrations are stable under composition.  The second
distinguishing feature of the internal language of model categories is
that the interpretation of $J$ terms need not satisfy the
``Beck-Chevalley'' condition --- traditionally assumed as part of
Martin-L\"{o}f type theory --- which states that, given
$\judge{v:A}{B(v):\type}$ and $c:A$ together
with the other hypotheses of the elimination rule,
\begin{align}\label{eq:coherence}
  \biggl(J_{B(v),D}\bigl(d(v),a(v),b(v),p(v)\bigr)\biggr)[c/v] & = J_{B(c),D}\bigl(d(c),a(c),b(c),p(c)\bigr).
\end{align}
The reason that (\ref{eq:coherence}) need not hold is that in
interpreting the $J$ term a choice of lift (\ref{eq:the_diagram}) is
made, and it may not, in general, be possible
to choose such lifts in a way which is compatible with pullback.
Nonetheless, there will always exists a (right)
homotopy between the interpretations of these terms and, in particular,
\begin{align*}
  \id{}\biggl(J_{B(v),D}\bigl(d(v),a(v),b(v),p(v)\bigr)[c/v],\;\;J_{B(c),D}\bigl(d(c),a(c),b(c),p(c)\bigr)\biggr)
\end{align*}
is always inhabited.  As such, the theory must be formulated either
as it is here, without requiring (\ref{eq:coherence}), or as a
form of dependent type theory with \emph{explicit substitution}
\cite{Abadi:ES,Curien:SUI}.

However, we believe that the failure of (\ref{eq:coherence}) to hold
constitutes a virtue, rather than a defect, of homotopy-theoretic
models.  Indeed, from the perspective of homotopy theory,
higher-dimensional category theory, and, indeed, mechanical
implementation of type theory, an internal language with some
(limited) form of explicit substitution is quite acceptable.  The
detailed syntax of this theory will be described in \cite{Warren:PhD}.

\subsection{Models satisfying the coherence condition}

Although the form of type theory modelled in all model categories and
finitely complete categories with weak factorization systems is
interesting in its own right, it is natural to consider models
satisfying the coherence condition (\ref{eq:coherence}).  A detailed
analysis of models satisfying (\ref{eq:coherence}) will be found in
\cite{Warren:PhD}; for now, we sketch one way to obtain such models.
In order to simplify the discussion we assume
the ambient category $\mathcal{C}$ is a cartesian closed model
category (or an appropriately enriched model category).  Then, if
$\mathcal{C}$ contains a unit interval $I$ satisfying certain basic
axioms such that exponentiation $A^{I}$ yields a path object for each
$A$, it is possible to define a (fibered) endofunctor
$T:\mathcal{C}\to<150>\mathcal{C}$ the pointed algebras of which are
distinguished fibrations called \emph{split fibrations} (and in many
cases $T$ will be a monad, although this is not strictly necessary).
Instead of interpreting types as fibrations we now interpret types as
\emph{split fibrations} in this sense.  Assuming that $I$ possesses
appropriate structure it is possible to choose lifts
(\ref{eq:the_diagram}) which satisfy (\ref{eq:coherence}).  For
example, the Hofmann-Streicher model in $\groupoids$ is obtained in
this way from the model structure.  It remains an open question whether it
is possible to prove a precise coherence (or strictification) theorem,
relating homotopy-theoretic models which do not satisfy (\ref{eq:coherence})
with models which do, analogous to the result of Hofmann
\cite{Hofmann:OITTLCCC} which, in a sense, solves the coherence issue
related to the interpretation extensional type theory in locally
cartesian closed categories.  
 
\subsection{Acknowledgements}

We would like to thank Andrej Bauer, Nicola Gambino, Andr\'{e} Joyal,
Per Martin-L\"{o}f and Alex Simpson for discussions of this material.
We also thank Erik Palmgren and Richard Garner for inviting us to
speak at the workshop ``Identity Types - Topological and Categorical
Structure'' held at Uppsala in November of 2006.  Finally, we give special
thanks to Ieke Moerdijk for suggesting this research topic and to
Thomas Streicher for many useful discussions.

\appendix

\section{The Syntax of Type Theory}\label{section:appendix}

The form of type theory validated as indicated in Theorem
\ref{theorem:main} consists of (\ref{eq:id_form})-(\ref{eq:id_conv})
together with the usual structural rules
(cf.~\cite{MartinLof:ITT,Nordstrom:PMLTT}) and the following
``Beck-Chevalley'' rules for the identity type and reflexivity terms:

\begin{align*}
  \begin{prooftree}
    \judge{x:C}{A(x):\type}\quad\judge{x:C}{a(x),b(x):A(x)}\quad \judge{}{c:C}
    \justifies
    \judge{}{\biggl(\id{A(x)}\bigl(a(x),b(x)\bigr)\biggr)[c/x]\;=\;\id{A(c)}\bigl(a(c),b(c)\bigr)\;:D\bigl(a(c),b(c),p(c)\bigr)}
    \using{\id{}\text{ B.-C.}}
  \end{prooftree}
\end{align*}

\begin{align*}
  \begin{prooftree}
    \judge{x:C}{A(x):\type}\quad\judge{x:C}{a(x):A(x)}\quad \judge{}{c:C}
    \justifies
    \judge{}{\biggl(r_{A(x)}\bigl(a(x)\bigr)\biggr)[c/x]\;=\;r_{A(c)}\bigl(a(c)\bigr):\id{A(c)}(a(c),a(c))}
    \using{r\text{ B.-C.}}
  \end{prooftree}
\end{align*}
\bibliographystyle{hplain}
\bibliography{short}
\end{document}